\documentclass[11pt]{amsart}

\usepackage{epsfig,overpic,comment}
\usepackage[usenames,dvipsnames,svgnames,table]{xcolor}
\usepackage[hyphens]{url}
\usepackage[pagebackref,linktocpage=true,colorlinks=true,linkcolor=Blue,citecolor=BrickRed,urlcolor=RoyalBlue]{hyperref}
\usepackage[msc-links,abbrev]{amsrefs}
\usepackage{amsmath,amsthm,amssymb,marginnote,upgreek}
\usepackage{enumerate}
\usepackage[T1]{fontenc}

\newtheorem{theorem}{Theorem}[section]

\newtheorem{lemma}[theorem]{Lemma}
\newtheorem{proposition}[theorem]{Proposition}

\theoremstyle{definition}
\newtheorem{note}[theorem]{Note}

\newcommand{\be}{\begin{equation}}
\newcommand{\ee}{\end{equation}}
\newcommand{\ol}{\overline}

\newcommand{\R}{\mathbf{R}}
\newcommand{\RP}{\mathbf{RP}}
\newcommand{\C}{\mathcal{C}}
\newcommand{\G}{\mathcal{G}}

\newcommand{\CAT}{\textup{CAT}}

\renewcommand{\S}{\mathbf{S}}
\renewcommand{\epsilon}{\varepsilon}
\renewcommand{\tilde}{\widetilde}

\DeclareMathOperator{\conv}{conv}

\DeclareMathOperator{\Hol}{Hol}
\DeclareMathOperator{\length}{length}
\DeclareMathOperator{\graph}{graph}
\DeclareMathOperator{\epi}{epi}

\makeatletter
\DeclareFontFamily{U}{tipa}{}
\DeclareFontShape{U}{tipa}{m}{n}{<->tipa10}{}
\newcommand{\arc@char}{{\usefont{U}{tipa}{m}{n}\symbol{62}}}%

\newcommand{\arc}[1]{\mathpalette\arc@arc{#1}}

\newcommand{\arc@arc}[2]{%
  \sbox0{$\m@th#1#2$}%
  \vbox{
    \hbox{\resizebox{\wd0}{\height}{\arc@char}}
    \nointerlineskip
    \box0
  }%
}
\makeatother

\begin{document}

\title[Total absolute curvature and rigidity of surfaces] 
{Total absolute curvature and rigidity of surfaces\\in Cartan-Hadamard manifolds}

\author[M. Ghomi]{Mohammad Ghomi}
\address{School of Mathematics, Georgia Institute of Technology,
Atlanta, GA 30332}
\email{ghomi@math.gatech.edu}
\urladdr{www.math.gatech.edu/~ghomi}

\author[J.A. Hoisington]{Joseph Ansel Hoisington}
\address{Dept. of Mathematics, Rose-Hulman Inst. of Technology,
Terre Haute, IN 47803}
\email{jhoisingt@rose-hulman.edu}

\author[M. Raffaelli]{Matteo Raffaelli}
\address{School of Mathematics, Georgia Institute of Technology, Atlanta, GA 30332}
\email{raffaelli@math.gatech.edu}
\urladdr{https://matteoraffaelli.com}

\author[J.I. Stavroulakis]{John Ioannis Stavroulakis}
\address{School of Mathematics, Georgia Institute of Technology,
Atlanta, GA 30332}
\email{jstavroulakis3@gatech.edu}

\begin{abstract}
We show that closed surfaces with minimal total absolute curvature in Cartan-Hadamard 3-manifolds bound flat convex bodies. This generalizes Chern-Lashof's theorem for surfaces in Euclidean space and solves a problem posed by Gromov in 1985. Our proof is based on an isometric embedding construction via holonomy, and uses Pogorelov’s theory of surfaces with bounded extrinsic curvature. Along the way, we obtain a regularity result for convex hulls  and a Schur-type comparison theorem for 
  curves in Cartan-Hadamard manifolds.
\end{abstract}

\date{\today \,(Last Typeset)}

\makeatletter
\def\subjclassname{\textup{2020} Mathematics Subject Classification}
\makeatother


\thanks{The first-named author was supported by NSF grant DMS-2202337.}

\maketitle


\section{Introduction}
A Cartan-Hadamard manifold $M^n$ is a complete simply connected Riemannian $n$-space with nonpositive curvature. The \emph{total absolute curvature} of a  closed  hypersurface $\Gamma$ immersed in $M$ is defined as 
$$
\tilde{\G}(\Gamma):=\int_\Gamma|GK|,
$$
where $GK$ is the Gauss-Kronecker curvature of $\Gamma$. Chern and Lashof \cite{chern-lashof1957,cecil-chern1985} showed that when $M$ is the Euclidean space $\R^n$, $\tilde{\G}(\Gamma)\geq |\S^{n-1}|$ with equality only if $\Gamma$ is convex, where $|\S^{n-1}|$ is the volume of the unit sphere in $\R^n$. 
We extend this result  to Cartan-Hadamard $3$-manifolds, 
as proposed by Gromov \cite[p.\ 66(b)]{ballmann-gromov-schroeder}:

\begin{theorem}\label{thm:main}
Let $\Gamma\subset M^3$ be a smooth closed immersed surface. Then 
\be\label{eq:main}
\tilde{\G}(\Gamma)\geq 4\pi,
\ee
with equality only if $\Gamma$ bounds a flat convex body.
\end{theorem}

By \emph{smooth} we mean $\C^\infty$. A \emph{convex body} in $M$  is a compact convex subset with interior points, and  is called \emph{flat} if the (sectional) curvature $K$ of $M$ vanishes on it. Schroeder and Strake \cite{schroeder-strake1989a}  proved Theorem \ref{thm:main} for strictly convex $\Gamma$. See also  \cite{ghomi-spruck2023a} for a refinement of that result by the first-named author and Spruck which incorporates upper bounds on $K$. More recently, Theorem \ref{thm:main} was established  for simply connected $\Gamma$ \cite{ghomi2025-convexity}, via techniques from Alexandrov geometry  outlined by Petrunin \cite{petrunin2022}. Here we refine those methods and devise other tools to settle the general case.

To prove Theorem \ref{thm:main}, we study the boundary $\Gamma_0$ of the convex hull of $\Gamma$. Kleiner observed that
$\tilde{\G}(\Gamma)\geq \tilde{\G}(\Gamma_0)$
\cite{kleiner1992,ghomi-spruck2022}, while the Gauss-Bonnet theorem together with Gauss' equation yields
$\tilde{\G}(\Gamma_0)\geq 4\pi$. Thus  inequality
\eqref{eq:main} is obtained relatively quickly. The main difficulty is characterizing the equality case, which involves a blend of metric and Riemannian geometry techniques to overcome the low  regularity of $\Gamma_0$.

For the equality case in \eqref{eq:main}, we first show via a result of Borb\'ely
\cite{borbely1995} that $\Gamma_0$ 
is $\C^1$ (Proposition~\ref{prop:convexhull}). 
Gauss' equation then forces the ambient curvature $K$ to vanish on tangent planes of $\Gamma_0$. This yields a
parallel frame along $\Gamma_0$, which we use to construct a $\C^1$
isometric embedding $f\colon \Gamma_0\to\R^3$. 
We show that 
$f$ preserves the total curvature of curves, and use Pogorelov’s theory of surfaces with bounded extrinsic curvature \cite{pogorelov1973} to prove that 
$f(\Gamma_0)$ is convex (Proposition~\ref{prop:C11}). 
Next we extend $f$ isometrically to the
entire convex hull (Proposition~\ref{prop:nested}) by combining the
techniques of \cite{ghomi2025-convexity}, namely the Kirszbraun-Lang-Schroeder extension  \cite{lang-schroeder1997} and Reshetnyak majorization \cite{reshetnyak1968} in $\CAT(0)$ spaces, with a new Schur-type
comparison theorem for $\C^1$ curves (Theorem~\ref{thm:schur}). The problem is
thereby reduced to the Euclidean case covered by 
Chern-Lashof's theorem.

Theorem~\ref{thm:main} naturally raises the question of whether the same phenomenon
persists in higher dimensions \cite[p.~66]{ballmann-gromov-schroeder}. In particular,
does the inequality $\tilde{\G}(\Gamma)\geq |\S^{n-1}|$ hold for convex hypersurfaces
$\Gamma\subset M^n$? This is a long-standing problem \cite{willmore-saleemi} which remains open for $n\geq 4$. An affirmative answer would imply the
Cartan--Hadamard conjecture on the extension of the classical isoperimetric
inequality to spaces of nonpositive curvature. See \cite{ghomi-spruck2022, kloeckner-kuperberg2019,ritore2023} for
background and references, and \cite{ghomi-spruck2023c,ghomi-stavroulakis2026,ghomi-stavroulakis2026b}
for more recent studies.

\section{Regularity of the Convex Hull}\label{sec:CH}
Throughout this work $M^n$ denotes an $n$-dimensional Cartan-Hadamard manifold, unless noted otherwise. Every pair of points of $M$ may be joined by a unique geodesic.
A subset of $M$ is \emph{convex} if it contains the geodesic connecting each pair of its points. The \emph{convex hull} of a set $X\subset M$, denoted by $\conv(X)$, is the intersection of all closed convex sets which contain $X$. We say that $\conv(X)$ is of regularity class $\C^{k,\alpha}$  if its boundary $\partial\conv(X)$ is a $\C^{k,\alpha}$ hypersurface. It is known that the convex hull of a closed $\C^{1,1}$ hypersurface in $\R^n$ is $\C^{1,1} $\cite[Note 6.8]{ghomi-spruck2022}; however, this fact has not been established in $M$.  The following weaker result will be sufficient for us:

\begin{proposition}\label{prop:convexhull}
Let $\Gamma\subset M^3$ be a closed topologically immersed surface that is differentiable at each point.  Then $\conv(\Gamma)$ is $\C^1$.
\end{proposition}

The conditions  above mean that $\Gamma$ is the image of a locally one-to-one continuous map $f\colon\ol\Gamma\to M$, for a closed $2$-manifold $\ol\Gamma$, and $f$ is differentiable at each point.
The proof uses basic facts from the theory of tangent cones \cite{ghomi-howard2014,ghomi-spruck2022}. For any set $X\subset\R^n$ and $p\in X$, the \emph{tangent cone} $T_p X$ of $X$ at $p$ is  the limit of all secant rays which emanate from $p$ and pass through a sequence of points of $X\setminus\{p\}$ converging to $p$. For a set $X\subset M$ and $p\in X$, the tangent cone is defined as
$
T_p X:= T_p\big(\exp_p^{-1}(X)\big)\subset T_p M,
$
where $\exp_p\colon T_p M\to M$ is the exponential map. We say that $T_p X$ is \emph{flat} if it is a hyperplane. A point of $X$ is \emph{extreme} if it lies on $\partial\conv(X)$. Since tangent cones of a differentiable surface are flat, the above proposition follows at once from:

\begin{theorem}\label{thm:convexhull}
Let $X\subset M^3$ be a compact set. Suppose that $\conv(X)$ has interior points and $T_p X$ is flat for all extreme points $p\in X$. Then $\conv(X)$ is  $\C^1$.
\end{theorem}

A \emph{convex hypersurface} $\Gamma\subset M$ is the boundary of a \emph{convex body}, i.e., a compact convex set with interior points. 
It is well known that a convex hypersurface in $\R^n$ is $\C^1$ if its  tangent cones are flat \cite{ghomi-howard2014}. We check that this fact holds in Cartan-Hadamard manifolds as well  via semiconvex functions \cite{cannarsa-sinestrari2004}. A function $f\colon\Omega\subset\R^{n-1}\to\R$ is \emph{semiconvex} provided that there exists $A\geq 0$ such that $x\mapsto f(x) +A|x|^2$ is convex.

\begin{lemma}\label{lem:convexhull}
A convex hypersurface $\Gamma\subset M^n$ with flat tangent cones is $\C^1$.
\end{lemma}
\begin{proof}
Each point of $\Gamma$ may be covered by a coordinate chart $(U,\phi)$ of $M$ such that $\phi(U\cap\Gamma)=\graph(f)$ for a semiconvex function $f\colon \Omega\subset\R^{n-1}\to\R$  \cite[Lem.\ 6.2]{ghomi-spruck2022}. The differential $d\phi$ maps tangent cones of $\Gamma\cap U$ to those of $\graph(f)$ and is a linear isomorphism at each point. 
Thus the tangent cones of $\operatorname{graph}(f)$ are flat. Since semiconvex functions
are locally Lipschitz, these hyperplanes cannot be vertical. Hence each tangent
cone of $\operatorname{graph}(f)$ is the graph of a linear function, and therefore $f$ is
differentiable at every point. Since $f$ is semiconvex, it follows that it is $\C^1$ \cite[Prop.\ 3.3.4]{cannarsa-sinestrari2004}.
\end{proof}

We also need the following result, which again follows from its analogue in $\R^n$ \cite[Lem.\ 5.7]{ghomi-howard2014}  via the same reduction technique used above.

\begin{lemma}\label{lem:boundary-tangent-cone}
Let $C\subset M^n$ be a convex body, and let $p\in \partial C$. Then
$
T_p(\partial C)=\partial(T_pC).
$
\end{lemma}

\begin{proof}
Let $(U,\phi)$ be a coordinate chart of $M$ around $p$ such that 
$
\phi(U\cap \partial C)=\graph(f),
$
for a semiconvex function $f\colon \Omega\subset \R^{n-1}\to \R$. Suppose  $\phi(p)=0=(0,f(0))$.
It is enough to show
$$
T_{0}\big(\graph (f)\big)=\partial\bigl(T_{0}(\epi (f))\bigr),
$$
where $\epi(f)$ stands for the epigraph of $f$.
Since $f$ is semiconvex, there exists $A\geq 0$ such that
$
g(x):=f(x)+A|x|^2
$
is convex near $0$. Define $F\colon\R^{n-1}\times\R\to\R^n$ by 
$
F(x,t):=(x,\; t+A|x|^2).
$
Then $F$ is a local diffeomorphism near $0$, with $F(0)=0$ and
$dF_{0}=I$, which maps $\graph(f)$ to $\graph(g)$ and $\epi(f)$ to $\epi(g)$. Thus
$$
T_{0}\big(\graph (f)\big)=T_{0}\big(\graph (g)\big), \qquad \partial\big(T_{0}(\epi (f))\big)=\partial\big(T_{0}(\epi (g))\big).
$$
But $\epi (g)$ is a convex set in $\R^n$ which is bounded by $\graph (g)$ near $0$. So by  \cite[Lem.\ 5.7]{ghomi-howard2014},
$
T_{0}(\graph (g))=\partial(T_{0}(\epi (g))),
$
which completes the proof.
\end{proof}

Finally we need the following result of Borb\'{e}ly \cite[Lem.\ 2.3]{borbely1995}:

\begin{lemma}[\cite{borbely1995}]\label{lem:borbely}
Let $X\subset M^3$ be a compact set, and $p\in\partial\conv(X)$. Suppose that there are no geodesics of $M$ with endpoints on $X$ which pass through $p$. Then $T_p (\partial\conv(X))$ is flat.
\end{lemma}

Now we are ready to establish the main result of this section:

\begin{proof}[Proof of Theorem \ref{thm:convexhull}]
Let $C:=\conv(X)$.
By Lemma \ref{lem:convexhull}, it suffices to show that $T_p(\partial C)$ is flat for all $p\in\partial C$. 

First suppose that $p\in\partial C\cap X$.  Since $C$ is convex, $T_p C$ is a proper convex cone \cite[Prop.\ 1.8]{cheeger-gromoll1972}. Furthermore, $T_pX$ is flat by assumption. Since $T_pX\subset T_p C$, it follows that $T_p C$ is one of the two closed half-spaces of $T_pM$ bounded by $T_p X$. Hence $T_pX=\partial(T_pC)$. But, by Lemma \ref{lem:boundary-tangent-cone}, $\partial (T_pC)=T_p(\partial C)$. Thus $T_p(\partial C)=T_pX$. In particular $T_p(\partial C)$ is flat, as desired.

It remains to consider the case where $p\in\partial C\setminus X$. By Lemma \ref{lem:borbely}, we may suppose that there exists a geodesic, say $\alpha$, which passes through $p$ with endpoints $q$, $q'\in X$.  Assume towards a contradiction that $T_p(\partial C)$ is not flat. 

Set $\ol C:=\exp_p^{-1}(C)$. Then again $T_p\ol C$ is a proper convex cone which 
contains $\ol C$ \cite[Prop.\ 1.8]{cheeger-gromoll1972}, and $T_p(\partial C)=\partial (T_p C)=\partial (T_p\ol C)$ by Lemma \ref{lem:boundary-tangent-cone} and the definition of tangent cone. Thus if $T_p(\partial C)$ is not flat, $\partial (T_p\ol C)$ is not flat either. Consequently there exist two distinct planes $\ol H$, $\ol H'\subset T_pM$ which support $\ol C$ at $p$. Then $H:=\exp_p(\ol H)$, $H':=\exp_p(\ol H')$  are complete surfaces in $M$ with respect to which $C$ lies on one side. 

Since $\alpha$ is a geodesic in $C$,  $\ol\alpha:=\exp_p^{-1}(\alpha)$ is a line segment in $\ol C$ passing through $p$.  It follows that $\ol\alpha\subset\ol H\cap\ol H'$, which yields  $\alpha\subset H\cap H'$. In particular $q\in H\cap H'$. Since $\ol H$ and $\ol H'$ are distinct, they are transversal along $\ol\alpha$. Thus, since $\exp_p$ is a diffeomorphism, $H$ and $H'$ are transversal along $\alpha$ and in particular at $q$. So $X$ admits distinct support surfaces at $q$. Hence $q\in\partial C$ and $T_qX$ is not flat, which is a contradiction.
\end{proof}

\begin{note}
The proof of Theorem \ref{thm:convexhull} carries over to higher dimensions
except for Lemma \ref{lem:borbely}. Borb\'ely \cite[p.\ 13]{borbely1995} constructed a counterexample to the
analogue of that lemma in $M^4$. However, his example
involves a set $X$ which does not have flat tangent cones, and therefore
does not rule out  higher-dimensional analogues of Theorem \ref{thm:convexhull}. See Lytchak and Petrunin \cite{lytchak-petrunin2022} for a recent study of the boundary structure of convex bodies in generic Riemannian manifolds.
\end{note}

\section{Isometric Embedding}
Let $\gamma\colon [0,\ell]\to M$ be a $\C^1$ unit speed curve.  Parallel transporting all tangent vectors $\gamma'(t)$ to $\gamma(0)$ along $\gamma$, we obtain a curve in the unit sphere $S_{\gamma(0)}M\subset T_{\gamma(0)}M$. The \emph{total curvature} $\tau(\gamma)$ may be defined as  the length of that curve.  If $\gamma$ is $\C^{1,1}$, then 
$$
\tau(\gamma)=\int_0^\ell\kappa\, dt,
$$ 
where $\kappa$ is the geodesic curvature.
Thus $\tau(\gamma)$ reduces to the usual notion of total curvature when $\gamma$ is sufficiently regular. In this section we show:

\begin{proposition}\label{prop:C11}
Let $\Gamma\subset M^3$ be a $\C^{1}$ convex surface. Suppose that
$K$ vanishes on tangent planes of $\Gamma$. Then there exists a $\C^{1}$ convex isometric  embedding $f\colon\Gamma\to \R^3$ that preserves the total curvature of all $\C^1$ curves in $\Gamma$.
\end{proposition}

In the case where $\Gamma$ is $\C^3$, the above result was established in \cite[Prop.\ 2.1]{ghomi2025-convexity} via the Gauss-Codazzi equations. The $\C^1$ case is considerably more involved. First we show that $TM$ has no holonomy along $\Gamma$  due to the condition on $K$ (Section \ref{subsec:holonomy}). Consequently there exists a parallel $\C^1$ frame on $\Gamma$ (Section \ref{subsec:frame}). Using this frame we construct a $\C^1$ isometric immersion $f\colon\Gamma\to\R^3$ which preserves the total curvature of curves (Section \ref{subsec:immersion}). Finally we use intrinsic and extrinsic curvature measures due to Alexandrov-Zalgaller and Pogorelov to show that $f$ is a convex embedding (Section \ref{subsec:convex}).

\subsection{Trivial holonomy}\label{subsec:holonomy}
Let $\nabla$ be the Levi-Civita connection on $M$ and $R$ be the Riemann curvature operator given by
$
R(X,Y)Z:=\nabla_X\nabla_Y Z-\nabla_Y\nabla_X Z-\nabla_{[X,Y]}Z,
$
for vector fields $X$, $Y$, $Z$ on $M$. 
The vanishing of $K$ on tangent planes of $\Gamma$, together with the fact that $K\leq 0$ in $M$, yields the following stronger curvature condition which was established in \cite[Lem.\ 2.2]{ghomi2025-convexity}:

\begin{lemma}[\cite{ghomi2025-convexity}]\label{lem:R}
For all $p\in\Gamma$, $X$ and $Y\in T_p\Gamma$, and $Z\in T_pM$, $R(X,Y)Z=0$.
\end{lemma}

This condition yields that $\Gamma$ is holonomically trivial in $M$. More precisely, for any pair of points $p$, $q\in M$ and $\C^1$ curve $\gamma$ connecting $p$ to $q$, let $P^\gamma_{p,q}\colon T_pM\to T_qM$ denote the isomorphism given by parallel translation along $\gamma$. Then we have the following result. The proof would have been quick if $\Gamma$ were smooth (see Note \ref{note:AS}), but the $\C^1$ case requires more care.

\begin{lemma}\label{lem:holonomy}
Let $p$, $q\in\Gamma$ and $\gamma_1$, $\gamma_2$ be $\C^1$ curves in $\Gamma$ connecting $p$ to $q$. Then 
$$
P^{\gamma_1}_{p,q}=P^{\gamma_2}_{p,q}.
$$
\end{lemma}

\begin{proof}
Since $\Gamma$ is simply connected,
there exists a $\C^1$ homotopy $H \colon [0,1]^2 \to \Gamma$ such that
$$
H(0,t)=\gamma_1(t), \qquad H(1,t)=\gamma_2(t), \qquad
H(s,0)=p, \qquad H(s,1)=q.
$$
Choose smooth curves $\gamma_i^k \colon [0,1]\to M$ with endpoints $p,q$ such that
$\gamma_i^k \to \gamma_i$ in $\C^1$.  There are smooth homotopies
$H^k \colon [0,1]^2 \to M$ such that
$$
H^k(0,t)=\gamma_1^k(t),\quad H^k(1,t)=\gamma_2^k(t), \quad H^k(s,0)=p,\quad H^k(s,1)=q,
$$
and $H^k \to H$ in $\C^1$. Set $H^k_t:=\partial_t H^k$ and $H^k_s:=\partial_sH^k$.
Pick a vector $V_0$ in the unit sphere $S_pM\subset T_pM$. For each $k$, let $V_k(s,t)$ be the parallel translate of $V_0$
along the curve $\alpha(t):=H^k(s,t)$. Then $\nabla_tV_k:=\nabla_{H_t^k}V_k=0$.
Since $H^k$ is smooth, so is $V_k$. Set $W_k=\nabla_sV_k:=\nabla_{H_s^k}V_k$.
Then
$$
\nabla_tW_k
=
\nabla_t\nabla_sV_k
=
\nabla_s\nabla_tV_k+R(H_t^k,H_s^k)V_k
=
R(H_t^k,H_s^k)V_k,
$$
since $[H_t^k,H_s^k]=0$. Furthermore, 
since $V_k(s,0)=V_0$, we have $W_k(s,0)=0$. 
Let
$
\ol W_k(t):=P^{\,\alpha}_{\alpha(t),\alpha(1)}W_k(s,t)\in T_{\alpha(1)}M.
$
Then
$
\ol W_k'(t)=P^{\,\alpha}_{\alpha(t),\alpha(1)}(\nabla_tW_k(s,t)).
$
Thus
$$
W_k(s,1)
=\ol W_k(1)-\ol W_k(0)
=\int_0^1 \ol W_k'(t)dt
=
\int_0^1
P^{\,\alpha}_{\alpha(t),\alpha(1)}\big(\nabla_t W_k(s,t)\big)\,dt.
$$
Hence, since parallel translation is an isometry,
$$
|W_k(s,1)|
\le
\int_0^1 |\nabla_tW_k(s,t)|\,dt
\le
\sup_{(s,t)\in[0,1]^2}\sup_{Z\in S_pM}|R(H_t^k,H_s^k)Z|=:C_k.
$$
Now, for every unit vector $V_0\in S_p M$, we obtain the uniform estimate
$$
\left|P^{\gamma_2^k}_{p,q}(V_0)-P^{\gamma_1^k}_{p,q}(V_0)\right|
=
|V_k(1,1)-V_k(0,1)|
\le
\int_0^1|W_k(s,1)|\,ds
\le
C_k.
$$
Since $H^k\to H$ in $\C^1$, we have $H_t^k\to H_t$ and $H_s^k\to H_s$ uniformly.
Furthermore, by Lemma~\ref{lem:R}, $R(H_t,H_s)Z=0$. So, by continuity of $R$,
$C_k\to 0$.
Consequently, the operator norm
$
|P^{\gamma_2^k}_{p,q}-P^{\gamma_1^k}_{p,q}|\to 0.
$
Since, by basic ODE theory, parallel transport depends continuously on the curve in the $\C^1$-topology, we also have
$|P^{\gamma_i^k}_{p,q}-P^{\gamma_i}_{p,q}|\to 0$. Therefore
$$
\big|P^{\gamma_1}_{p,q}-P^{\gamma_2}_{p,q}\big|
\leq 
\big|P^{\gamma_1}_{p,q}-P^{\gamma_1^k}_{p,q}\big|
+
\big|P^{\gamma_1^k}_{p,q}-P^{\gamma_2^k}_{p,q}\big|
+
\big|P^{\gamma_2^k}_{p,q}-P^{\gamma_2}_{p,q}\big|\to 0,
$$
 which completes the proof.
\end{proof}

\begin{note}\label{note:AS}
When $\Gamma$ is smooth, Lemma \ref{lem:holonomy} follows quickly from the Ambrose-Singer holonomy theorem  \cite[Thm.\ 2]{ambrose-singer1953} \cite[Thm.\ 3.1.22]{ballmann2022}.
Consider the vector bundle
$TM|_\Gamma\to \Gamma$ with the connection induced from  $M$.  Let $\Hol_p^\Gamma(M)$ denote its holonomy group at $p\in\Gamma$. Since $\Gamma$ is simply connected, 
 the Ambrose-Singer theorem yields that the associated Lie algebra  is
generated by the endomorphisms
$$
P^{-\gamma}_{q,p}\circ R(X,Y)\circ P^{\,\gamma}_{p,q},
$$
where $\gamma$ is a $\C^1$ curve in $\Gamma$ from $p$ to another point $q\in \Gamma$, $-\gamma$ indicates the reverse parametrization of $\gamma$, and
$X,Y\in T_q\Gamma$. By Lemma \ref{lem:R}, $R(X,Y)Z=0$ for all $Z\in T_qM$. Thus $\Hol_p^{\Gamma}(M)$ is trivial. In particular, letting $\gamma:=\gamma_1\ast(-\gamma_2)$ be the concatenation of the curves  in Lemma \ref{lem:holonomy} completes the proof. The proof of Lemma \ref{lem:holonomy} basically confirms that the Ambrose-Singer theorem  holds in the $\C^1$ category, by extending  the standard argument with an approximation.
\end{note}

\subsection{The parallel frame}\label{subsec:frame}
A vector field $V\colon \Gamma\to TM$ is \emph{parallel} along $\Gamma$ provided that $\nabla_X V=0$ for all tangent vectors $X\in T\Gamma$. Using the triviality of the holonomy established above, we now construct a parallel frame on $\Gamma$. Since $\Gamma$ is only $\C^1$, it requires some care to show that the frame is $\C^1$ as well.

\begin{lemma}\label{lem:ei}
There exists a parallel orthonormal $\C^1$ frame field $(e_1,e_2,e_3)$ along $\Gamma$.
\end{lemma}
\begin{proof}
Fix $p_0\in\Gamma$, and choose an orthonormal basis $e_i$ for $T_{p_0} M$. For any point $p$ of $\Gamma$ we parallel translate $e_i$ to $p$ along a $\C^1$ curve in $\Gamma$ connecting $p_0$ to $p$. By Lemma \ref{lem:holonomy}, the choice of the curve is immaterial. Thus we obtain an orthonormal frame on $\Gamma$. 

By path independence, the restriction of $e_i$ to any $\C^1$ curve is obtained by parallel transport along that curve. Hence, for every $X\in T_p\Gamma$, if $\alpha$ is a curve with $\alpha'(0)=X$, then $\nabla_X e_i=0$. Thus $e_i$ is parallel on $\Gamma$.

It remains to show that $e_i$ is $\C^1$ on $\Gamma$. 
First note that since $e_i$ is parallel along $\C^1$ curves, it is the solution along those curves to an ODE with coefficients given by the Christoffel symbols, which are continuous. Hence $e_i$ is $\C^1$ along any $\C^1$ curve.

Let $\varphi\colon U\subset\R^2\to \Gamma$
be a local $\C^1$ parametrization, and  $x_i$ be local coordinates of $M$ on an open set containing $\varphi(U)$. We will show that $e_i\circ\varphi$ is $\C^1$. Fix $u_0\in U$, and for $u$ near $u_0$,
join $\varphi(u_0)$ to $\varphi(u)$ by the curve
$
t\mapsto \varphi((1-t)u_0+tu).
$
These curves depend continuously on $u$ in the $\C^1$ topology. Since $e_i\circ\varphi(u)$ is obtained by parallel transport
of  $e_i\circ\varphi(u_0)$ along these curves, $e_i\circ\varphi$
is continuous.
Now write (using Einstein's summation convention)
$$
e_i\circ\varphi(u)=a_{ij}(u)\,\partial_j\big|_{\varphi(u)},
$$
where $\partial_j=\partial/\partial x_j$.
Note that  $a_{ij}$ are continuous, because $e_i\circ\varphi(u)$ and $\partial_j\big|_{\varphi(u)}$ both depend continuously on $u$.
Let 
$
\gamma_k(t):=\varphi(u+te'_k)
$
be the coordinate curves on $\Gamma$, where $e'_k$ is the standard basis for $\R^2$. 
Since $e_i$ is $\C^1$ on $\gamma_k$, the functions $t\mapsto a_{ij}(u+te'_k)$ are $\C^1$. So we may compute that
$$
0=\nabla_{\gamma_k'(t)}e_i
=\nabla_{\gamma_k'(t)}
\Big(a_{ij}(u+te'_k)\,\partial_j\Big)
=\frac{d}{dt}a_{ij}(u+te'_k)\,\partial_j
+a_{i\ell}(u+te'_k)\,\nabla_{\gamma_k'(t)}\partial_\ell.
$$
Since
$
\gamma_k'=(d(x_m\circ \gamma_k)/dt)\partial_m
$
and
$
\nabla_{\partial_m}\partial_\ell=\Upgamma^j_{\ell m}\,\partial_j,
$
it follows that
$$
\frac{d}{dt}a_{ij}(u+te'_k)
=
-\Upgamma^j_{\ell m}(\gamma_k(t))
\frac{d(x_m\circ \gamma_k)}{dt}(t)\,
a_{i\ell}(u+te'_k).
$$
Evaluating at $t=0$, we obtain
$$
\partial_ka_{ij}(u)
=
-\Upgamma^j_{\ell m}(\varphi(u))
\,\partial_k(x_m \circ\varphi)(u)\,
a_{i\ell}(u).
$$
The right-hand side is continuous, since the Christoffel symbols are smooth,
$\varphi$ is $\C^1$, and $a_{i\ell}$ is continuous. Hence each $a_{ij}$ is $\C^1$, and
therefore $e_i$ is $\C^1$ as desired.
\end{proof}

There are functions which are $\C^1$ on every $\C^1$ curve  but are not $\C^1$ \cite[Thm.\ 3]{boman1967}. Thus the discussion in the second part of the above proof is not superfluous.

\subsection{Isometric immersion}\label{subsec:immersion}
Extend the frame $e_i$, given by Lemma \ref{lem:ei}, to a $\C^1$ orthonormal frame on an open neighborhood $U$ of $\Gamma$ in $M$. 
Let $\theta_i(\cdot):=\langle\cdot,e_i\rangle$ be the dual coframe. Since $e_i$ is $\C^1$ on $U$, so is $\theta_i$. Fix $p_0\in\Gamma$. For any $p\in\Gamma$ let $\gamma\colon[0,\ell]\to\Gamma$ be a $\C^1$ unit speed curve connecting $p_0$ to $p$. Then define $f=(f_1,f_2,f_3)\colon\Gamma\to\R^3$ by
$$
f_i(p):=\int_{\gamma}\theta_i.
$$
First we check that $f$ does not depend on the choice of $\gamma$. Note that for tangent vector fields $X$, $Y$ on $\Gamma$,
\begin{eqnarray*}
d\theta_i(X,Y) &=& X(\theta_i(Y))-Y(\theta_i(X))-\theta_i([X,Y])\\
&=& X\langle Y, e_i\rangle-Y\langle X, e_i\rangle-\langle[X,Y],e_i\rangle\\
&=& \langle \nabla_X Y, e_i\rangle+\langle Y,\nabla_Xe_i\rangle-\langle \nabla_Y X, e_i\rangle-\langle X,\nabla_Ye_i\rangle-\langle[X,Y],e_i\rangle\\
&=&\langle \nabla_X Y-\nabla_Y X-[X,Y],e_i\rangle=0.
\end{eqnarray*}
 Now let $\gamma_1$, $\gamma_2$ be $\C^1$ curves connecting $p_0$ to $p$. Since $\Gamma$ is simply connected, there exists a domain $\Sigma\subset\Gamma$, possibly with multiple components, such that $\partial\Sigma=\gamma_1\cup-\gamma_2$. So by Stokes' theorem $\int_{\gamma_1}\theta_i-\int_{\gamma_2}\theta_i=\int_\Sigma d\theta_i=0$, as desired. Next we check that
 $$
 df_i=\theta_i.
 $$
 Since $d\theta_i=0$, this would have followed from the Poincar\'{e} lemma if $\Gamma$ were smooth, but again we need a direct argument.
 Fix $p\in\Gamma$, and let $\sigma\colon(-\varepsilon,\varepsilon)\to\Gamma$ be a $\C^1$ curve with $\sigma(0)=p$. Choose a $\C^1$ curve $\gamma_0$ in $\Gamma$ connecting $p_0$ to $p$, and for each $t$ let $\gamma_t$ be the concatenation of $\gamma_0$ with $\sigma|_{[0,t]}$. Then, by path-independence,
$$
f_i\big(\sigma(t)\big)=\int_{\gamma_t}\theta_i
=\int_{\gamma_0}\theta_i+\int_{\sigma|_{[0,t]}}\theta_i.
$$
Thus
$
f_i(\sigma(t))-f_i(\sigma(0))
=\int_{\sigma|_{[0,t]}}\theta_i.
$
Now, since $\theta_i$ is continuous, the fundamental theorem of calculus yields that
$$
df_i\big(\sigma'(0)\big)=\frac{d}{dt}\Big|_{t=0} f_i\circ\sigma(t)=\theta_i(\sigma'(0)).
$$
So $df_i=\theta_i$, as claimed.
Using this property, we obtain the next two lemmas.
\begin{lemma}\label{lem:isometry}
$f$ is a $\C^1$ isometric immersion.
\end{lemma}
\begin{proof}
Since 
$df_i=\theta_i$ and $\theta_i$ is continuous,  each $f_i$ is $\C^1$.  Moreover, for any $X,Y\in T_p\Gamma$,
$$
\big\langle df_p(X),df_p(Y)\big\rangle_{\R^3}
=df_i(X)\,df_i(Y)
=\theta_i(X)\,\theta_i(Y)
=\langle X,Y\rangle_M,
$$
since $e_i$ is an orthonormal frame, which completes the proof.
\end{proof}

\begin{lemma}
 $f$ preserves the total
curvature of $\C^1$ curves in $\Gamma$.
\end{lemma}
\begin{proof}
Let $\gamma\colon [0,\ell]\to \Gamma$ be a $\C^1$ unit speed curve, and write
$
\gamma'(t)=a_i(t)e_i(\gamma(t)).
$
Since $e_i$ is parallel along $\Gamma$, the parallel transport of $\gamma'(t)$
to $T_{\gamma(0)}M$ along $\gamma$ yields the curve
$
t\mapsto a_i(t)e_i(\gamma(0)),
$
whose length equals $\tau(\gamma)$.
On the other hand,
$$
(f\circ\gamma)'(t)=df_{\gamma(t)}(\gamma'(t))
= \theta_i(\gamma'(t))\,e'_i
= a_i(t)e'_i,
$$
where $e'_i$ is the standard basis of $\R^3$. So $\tau(f\circ\gamma)$ is the length of the curve
$
t\mapsto  a_i(t)e'_i.
$
Thus $\tau(f\circ\gamma)=\tau(\gamma)$.
\end{proof}

\subsection{Convexity}\label{subsec:convex}
Here we show that $\Gamma' := f(\Gamma)$ is a convex surface. Consequently, by a covering argument,
$f$ is an embedding,  which completes the proof of Proposition \ref{prop:C11}.
To prove the convexity of $\Gamma'$ we show that the intrinsic curvature of $\Gamma'$ in the sense of Alexandrov-Zalgaller is nonnegative. Furthermore, the extrinsic curvature  in the sense of Pogorelov is bounded. It follows then from Pogorelov's generalization of Gauss' Theorema Egregium that $\Gamma'$ has nonnegative extrinsic curvature measure, and hence is convex.

\subsubsection{Nonnegativity of the intrinsic curvature}
The \emph{intrinsic curvature measure} of a $\C^1$ surface  $\Gamma$, 
   due to Alexandrov-Zalgaller \cite{alexandrov-zalgaller},  
   is  a signed Radon measure $\omega$ defined via excess angles of geodesic triangles $\Delta\subset \Gamma$. More precisely, if the interior angles of $\Delta$ are $\alpha,\beta,\gamma$, then its excess angle is
$
e(\Delta):=\alpha+\beta+\gamma-\pi.
$
For a region $U\subset\Gamma$, one defines
$$
\omega^+(U):=\sup \sum_i e(\Delta_i)^+,
\qquad
\omega^-(U):=\sup \sum_i e(\Delta_i)^-,
$$
where the supremum is taken over all finite families of pairwise nonoverlapping geodesic triangles $\Delta_i\subset U$, and
$
x^+:=\max\{x,0\},
$
$
x^-:=\max\{-x,0\}.
$
Then
$
\omega:=\omega^+-\omega^-.
$
In particular, $\omega\geq 0$ if $e(\Delta)\geq 0$ for all geodesic triangles $\Delta\subset\Gamma$. 
 Furthermore, when $\Gamma$ is smooth, 
$
\omega(U)=\int_U K_\Gamma,
$
where $K_\Gamma$ is the (intrinsic) Gauss curvature of $\Gamma$. We need the following  fact:

\begin{lemma}\label{lem:Gamma-i}
Let $\Gamma\subset M^n$ be a $\C^1$ convex hypersurface. Then there exists a sequence of smooth convex hypersurfaces $\Gamma_i\subset M^n$ which converges to $\Gamma$ in $\C^1$-topology.
\end{lemma}
\begin{proof}
Let $u\colon M\to\R$ be the distance function of the convex body $C$ bounded by $\Gamma$. Then $u$ is $\C^1$ on $M\setminus C$ \cite[Lem.\ 2.3]{ghomi-spruck2022}. Set 
$
\ol u_{\epsilon}(x):=u(x)+\epsilon \text{dist}^2(x,x_0)/2,
$
for some point $x_0\in C$ and $\epsilon>0$.
Then $\ol u_{\epsilon}$ is strictly convex in the sense of Greene and Wu,  and thus their Riemannian convolution preserves convexity \cite{greene-wu1972} \cite[Prop.\ 4.8]{ghomi-spruck2022}. More specifically, let $\phi\colon \R \to \R$ be a nonnegative $\C^{\infty}$ function supported in $[-1,1]$ which is constant in a neighborhood of the origin, and satisfies $\int_{\R^n}\phi(|x|)dx=1$. Then 
$$
\tilde u_\epsilon(x):=\frac{1}{\epsilon^{n}}\int_{v\in T_xM} \phi\left(\frac{|v|}{\epsilon}\right) \ol u_{\epsilon}\big(\exp_x(v)\big),
$$
yields a family of smooth strictly convex functions such that $\tilde u_\epsilon\to u$ in $\C^{1}$ on compact subsets of $M\setminus C$, as $\epsilon\to 0$. 
Thus
$
\Gamma_{i,\epsilon}:=\tilde u_{\epsilon}^{-1}(1/i)
$
are smooth convex hypersurfaces for $i\in\mathbf{N}$.
As $\epsilon\to 0$, $\Gamma_{i,\epsilon}\to u^{-1}(1/i)$ in $\C^1$. Furthermore, $u^{-1}(1/i)\to\Gamma$ in $\C^1$ as $i\to\infty$. Hence we may choose a sequence $\epsilon_i\to 0$ such that $\Gamma_i:=\Gamma_{i,\epsilon_i}\to \Gamma$ in $\C^1$.
\end{proof}

The approximation above, together with stability under Gromov-Hausdorff convergence of Alexandrov spaces with curvature bounded below  \cite{bbi2001}, yields that: 

\begin{lemma}\label{lem:intrinsic}
The intrinsic curvature of $\Gamma$ is nonnegative.
\end{lemma}
\begin{proof}
Let $\Gamma_i$ be the smooth approximations of $\Gamma$ provided by Lemma \ref{lem:Gamma-i}. Since $\Gamma_i\to\Gamma$ in $\C^1$ and $K$ vanishes on tangent planes of $\Gamma$, it follows that 
$$
\epsilon_i:=\sup_{p\in \Gamma_i}|K(T_p\Gamma_i)|\to 0,
$$
by continuity of $K$ on the Grassmannian of planes in $TM$ and compactness of $\Gamma$.
Since $\Gamma_i$ is convex, its Gauss-Kronecker curvature $GK\geq 0$. Thus, by Gauss' equation, 
$$
K_{\Gamma_i}(p)=K(T_p\Gamma_i)+GK(p)\geq K(T_p\Gamma_i)\geq -\epsilon_i.
$$
So $(\Gamma_i,d_{\Gamma_i})$ are Alexandrov spaces with curvature bounded
below by $-\varepsilon_i$. But $(\Gamma_i,d_{\Gamma_i})\to
(\Gamma,d_\Gamma)$ in the Gromov-Hausdorff sense \cite[Thm.\ 10.2.7 \& Ex.\ 7.4.2]{bbi2001}. It follows that $(\Gamma,d_\Gamma)$ is an Alexandrov space with curvature bounded below by $0$ \cite[Prop.\ 10.7.1]{bbi2001}. Thus the excess $e(\Delta)\geq 0$ for all geodesic triangles $\Delta\subset\Gamma$, which yields that $\omega\geq 0$.
\end{proof}

By Alexandrov's embedding theorem \cite{alexandrov2005}, the last lemma implies that there exists a convex isometric embedding $\Gamma\to\R^3$. Since the induced metric on $\Gamma$ is $\C^1$, one might expect that the embedding will be $\C^1$ as well; however, this is not known. The closest regularity result is due to Guan and Li \cite{guan-li1994}, who showed that if the metric is $\C^4$ then the embedding is $\C^{1,1}$; see also \cite{guan-li1997}.

\subsubsection{Boundedness of the extrinsic curvature}
Since $f$ is only $\C^1$, the nonnegativity of the intrinsic curvature of $\Gamma$ does not imply that $\Gamma'$ is convex. Indeed the Nash-Kuiper theorem \cite{nash1954,kuiper1955,eliashberg-mishachev2024} shows that there is no connection between the intrinsic and extrinsic geometry of $\C^1$ surfaces in $\R^3$. Thus we need to extract more information from the construction of $f$:

\begin{lemma}\label{lem:nuprime}
There exists a unit normal vector field $\nu'$ along $\Gamma'$ such that for all $p\in \Gamma$
$$
\big\langle \nu(p),e_i(p)\big\rangle_M=\big\langle \nu'(f(p)),e_i'\big\rangle_{\R^3},
$$
where $\nu$ is the outward unit normal of $\Gamma$, and $e_i'$ is the standard basis of $\R^3$. 
\end{lemma}
\begin{proof}
Let $\nu_i(p):=\langle \nu(p),e_i(p)\rangle_M$,
and set
$\nu'(f(p)):=\nu_i(p)\,e_i'.$
Then
$
\langle \nu(p),e_i(p)\rangle_M=\langle \nu'(f(p)),e_i'\rangle_{\R^3},
$
and $|\nu'(f(p))|=|\nu(p)|=1$ as desired. It remains to check that $\nu'$ is normal to $\Gamma'$. For any $X\in T_p\Gamma$,
$
df_p(X)=\theta_i(X)\,e_i'
=
 \langle X,e_i(p)\rangle_M \,e_i'.
$
Thus
$
\langle df_p(X),e_i'\rangle_{\R^3}
=
\langle X,e_i(p)\rangle_M.
$
It follows that
$$
\big\langle df_p(X),\nu'(f(p))\big\rangle_{\R^3}
=
 \nu_i(p)\langle X,e_i(p)\rangle_M
=
\langle X,\nu(p)\rangle_M
=
0,
$$
which completes the proof. 
\end{proof}

A closed $\C^1$ oriented surface $\Gamma\subset\R^3$ has \emph{bounded extrinsic curvature}, in the sense of Pogorelov \cite[p.\ 590]{pogorelov1973} \cite[Def.\ 2.1]{pakzad2024}, if its Gauss map has bounded variation or is $BV$, i.e., for any compact set $X\subset\Gamma$, $\nu(X)\subset\S^2$ has finite volume. Note that in the terminology of Pogorelov ``smooth'' means $\C^1$ 
\cite[p.\ 572]{pogorelov1973}.

\begin{lemma}\label{lem:bounded}
$\Gamma'$ has bounded extrinsic curvature.
\end{lemma}
\begin{proof}
The outward normal $\nu'$ of $\Gamma'$ is BV if its components $\nu_i':=\langle \nu',e_i'\rangle_{\R^3}$ are BV. By Lemma \ref{lem:nuprime}, $\nu_i'\circ f=\nu_i:=\langle \nu, e_i\rangle_M$. So, since $f\colon\Gamma\to\Gamma'$ is  locally bi-Lipschitz, it suffices to check that $\nu_i$ is BV. This in turn follows if $\nu$ is BV in local coordinates, since $e_i$ is $\C^1$. Choose a local chart
$
(U, \phi)
$
of $M$
such that
$$
\phi(\Gamma\cap U)
=
\big\{(x,z)\in \Omega\times\R\mid\; z=u(x)\big\},
$$
where  $u\in \C^{1}(\Omega)$.
Since $\Gamma$ is convex, $u$ is semiconvex \cite[Lem.\ 6.2]{ghomi-spruck2022}, and therefore
$
\nabla u \in BV_{\mathrm{loc}}(\Omega;\R^{2}).
$
Set
$
F(x,z):=z-u(x)
$
so that $\phi(\Gamma\cap U)=F^{-1}(0)$.
Let $g=(g_{ij})$ denote the metric of $M$.
The Riemannian gradient is given by
$
\nabla^{g}F=g^{-1}\cdot\, \nabla F.
$
Thus 
$$
\nu
=
\frac{\nabla^{g}F}{|\nabla^{g}F|_{g}}
=
\frac{g^{-1}\,(-\nabla u,1)}
{\sqrt{(-\nabla u,1)^{T}
\,g^{-1}\,(-\nabla u,1)}}.
$$
So $\nu(\cdot,u(\cdot)) = H(\cdot,\nabla u(\cdot))$, where
$H\colon\Omega\times\R^2\to\R^3$ is $\C^1$, hence locally Lipschitz.
Since $\nabla u\in BV_{\mathrm{loc}}(\Omega;\R^2)$, it follows that
$\nu\in BV_{\mathrm{loc}}(\Omega;\R^3)$, as desired.
\end{proof}

\subsubsection{Gauss-Pogorelov Theorema Egregium}
Using the findings above,  together with two more results of Pogorelov for $\C^1$ surfaces with bounded extrinsic curvature, we now obtain the following lemma which concludes the proof of Proposition \ref{prop:C11}. 

The main tool we apply here is Pogorelov's remarkable generalization of Gauss' Theorema Egregium \cite[Chap.\ IX, Sec.\ 9]{pogorelov1973} 
to surfaces with bounded extrinsic curvature, which equates the intrinsic and extrinsic curvature measures; see  \cite[p.\ 1138]{borisenko2019} and \cite[Sec.\ 6]{borisenko2008} for overviews. The \emph{extrinsic curvature measure} $\sigma$ is defined  by generalizing the notions of elliptic and hyperbolic points to $\C^1$ surfaces $\Gamma\subset\R^3$. For any region $U\subset\Gamma$, let $\sigma^+(U)$ and $\sigma^-(U)$ denote the total  variation of the Gauss map over the elliptic and hyperbolic points respectively. Then  $\sigma(U):=\sigma^+(U)-\sigma^-(U)$ \cite[p.\ 593]{pogorelov1973}. The generalized Theorema Egregium  states that
$$
\sigma^+(U)=\omega^+(U), \qquad \sigma^-(U)=\omega^-(U),
$$
for all Borel sets $U\subset\Gamma$ \cite[Thm.\ 4, p.\ 649]{pogorelov1973}. Thus $\sigma=\omega$.

The other result of Pogorelov we need is that a closed surface of bounded extrinsic curvature in $\R^3$ with nonnegative extrinsic curvature  is convex \cite[Thm.\ 2, p.\ 615]{pogorelov1973}, which is a generalization of the Chern-Lashof characterization of $\C^2$ convex surfaces \cite{chern-lashof1957}.

\begin{lemma}
$\Gamma'$ is convex.
\end{lemma}
\begin{proof}
By Lemma \ref{lem:intrinsic} the intrinsic curvature  of $\Gamma$ is nonnegative. Consequently, the intrinsic curvature  of $\Gamma'$ is nonnegative as well by isometry.  Since, by Lemma \ref{lem:bounded}, $\Gamma'$ has bounded extrinsic curvature, it follows from the generalized Theorema Egregium \cite[Thm.\ 4, p.\ 649]{pogorelov1973}  that $\Gamma'$ has nonnegative extrinsic curvature. Consequently  $\Gamma'$ is convex \cite[Thm.\ 2, p.\ 615]{pogorelov1973}.
\end{proof}

\section{A Schur-type Comparison Theorem}
Here we generalize the classical Schur comparison theorem \cites{chern1967,sullivan2008}  for curves in $\R^n$, also known as the ``bow lemma'' \cite{petrunin-zamora2022,gromov2025}, to Cartan-Hadamard manifolds.
A curve $\gamma\colon[0,\ell]\to\R^2$ is \emph{chord-convex} if connecting its endpoints  $\gamma(0)$ and $\gamma(\ell)$ 
yields a simple closed curve which bounds a convex body.
For every pair of points $p$, $q\in M$ let $pq$ denote the geodesic connecting them, and $|pq|$ be the length of $pq$. 

\begin{theorem}\label{thm:schur}
Let $\gamma_1\colon [0,\ell]\to\R^2$, $\gamma_2\colon [0,\ell]\to M^n$ be $\C^{1}$ unit speed curves.
Suppose that $\gamma_1$ is chord-convex, and for every interval $I\subset [0,\ell]$, $\tau(\gamma_2(I))\leq \tau(\gamma_1(I))$. Then 
$$
|\gamma_2(0)\gamma_2(\ell)|\geq |\gamma_1(0)\gamma_1(\ell)|.
$$
\end{theorem}

A partial extension of Schur's theorem to hyperbolic space $\mathbf{H}^n$ was developed by Epstein \cite{epstein1985}, and the polygonal version, known as Cauchy's ``arm lemma'' \cite{aigner-ziegler1999}, holds in $\CAT(0)$ spaces \cite[9.63]{akp2024}. 
The above result was established in \cite[Thm.\ 3.1]{ghomi2025-convexity} in the case where $\gamma_1$ and $\gamma_2$ are $\C^{2}$, via polygonal approximations after Epstein \cite{epstein1985} and estimates of Alexander-Bishop \cite{alexander-bishop1996}. Here we adopt a more analytic approach. See \cite{gromov2025,ni2023,mendes2025,howard2024,lopez2011} for other recent variations of Schur's result and applications.

We prove Theorem \ref{thm:schur} by reducing it to the Euclidean case where the classical Schur theorem  finishes the argument. This is achieved via majorization in the sense of Reshetnyak (Section \ref{subsec:majorization}), once we check that  this operation does not increase curvature. To this end we develop several estimates for the chord-lengths of curves in Cartan-Hadamard manifolds (Section \ref{subsec:estimates}). These estimates yield the theorem in the $\C^{1,1}$ case (Section \ref{subsec:C11}), and an approximation argument completes the proof (Section \ref{subsec:proof}).

We will assume that all curves below  have unit speed, i.e., are parametrized by arclength, unless stated otherwise.

\subsection{Majorization}\label{subsec:majorization}
Let $\gamma\colon[0,\ell]\to M$ be a curve. We say that a curve $\tilde\gamma\colon[0,\ell]\to\R^2$ \emph{majorizes} $\gamma$ provided that  
 $
|\tilde\gamma(t)\tilde\gamma(s)|\geq|\gamma(t)\gamma(s)|
$ 
for all $t$, $s\in[0,\ell]$. In addition if $ |\tilde\gamma(0)\tilde\gamma(\ell)|=|\gamma(0)\gamma(\ell)|$, then we say that $\tilde\gamma$ is \emph{proper}. The majorizing curve is also known as ``unfolding''  \cites{cks2002,ghomi-wenk2021} or ``chord-stretching'' \cites{sallee1973,brooks-strantzen1992}. Reshetnyak's  theorem \cites{akp2024,reshetnyak1968} states that when $\gamma$ is closed, i.e., $\gamma(0)=\gamma(\ell)$,  it admits a proper majorization by a convex curve. We need the following variation:

 \begin{lemma}\label{lem:reshetnyak}
 For every $\C^1$ curve $\gamma\colon[0,\ell]\to M$  there exists a chord-convex $\C^1$ curve $\tilde\gamma\colon[0,\ell]\to \R^2$ which properly majorizes  $\gamma$.
  \end{lemma}
  \begin{proof}
  Join $\gamma(\ell)$ to $\gamma(0)$  by a geodesic to obtain the extension of $\gamma$ to a closed constant speed curve $\gamma_0\colon[0,\ell_0]\to M$, where $\ell_0:=\ell+|\gamma(0)\gamma(\ell)|$. 
  By Reshetnyak's theorem, $\gamma_0$ is majorized by a convex planar curve $\tilde\gamma_0\colon [0,\ell_0]\to\R^2$. Any majorization preserves the geodesic subsegments of the curve \cite[Prop.\ 9.54]{akp2024}. In particular $\tilde\gamma_0([\ell,\ell_0])$ is a line segment with  length equal to $|\gamma_0(\ell)\gamma_0(\ell_0)|=|\gamma(0)\gamma(\ell)|$. Thus the restriction of $\tilde\gamma_0$ to $[0,\ell]$ yields a proper majorizing  curve $\tilde\gamma$. 
  
Next we show that $\tilde\gamma$ is $\C^1$. First note that one-sided derivatives of $\tilde\gamma$ exist at each point by convexity \cite[Thm.\ 1.5.4]{schneider2014}. Furthermore, for each $t_0\in(0,\ell)$, 
$$
|\tilde\gamma(t_0-h)\tilde\gamma(t_0+h)|
\ge |\gamma(t_0-h)\gamma(t_0+h)|
=2h+o(h),
$$
since $\gamma$ is differentiable at $t_0$.
On the other hand, 
$
|\tilde\gamma(t_0-h)\tilde\gamma(t_0+h)|\le 2h,
$
by the triangle inequality.
So
$
|\tilde\gamma(t_0-h)\tilde\gamma(t_0+h)|=2h+o(h).
$
Next note that 
$$
2h+o(h)=|\tilde\gamma(t_0-h)\tilde\gamma(t_0+h)|
\le
|\tilde\gamma(t_0-h)\tilde\gamma(t_0)|
+
|\tilde\gamma(t_0)\tilde\gamma(t_0+h)|
\le 2h.
$$
Therefore
$|\tilde\gamma(t_0-h)\tilde\gamma(t_0)|=h+o(h)$ and
$|\tilde\gamma(t_0)\tilde\gamma(t_0+h)|=h+o(h)$. Now
the law of cosines  implies that
the angle 
$$
\measuredangle\big(\tilde\gamma(t_0-h),\tilde\gamma(t_0),\tilde\gamma(t_0+h)\big)\to\pi.
$$
So the left and right derivatives of $\tilde\gamma$ coincide at $t_0$, which shows that $\tilde\gamma$  is differentiable everywhere. Finally, 
  since $\tilde\gamma$ is convex, it follows that it is $\C^1$ \cite[Thm.\ 1.5.4]{schneider2014}.
    \end{proof}

To prove Theorem \ref{thm:schur}, let $\tilde\gamma_2$ be a proper majorization of $\gamma_2$ furnished by the above lemma. It suffices then to show that $|\tilde\gamma_2(0)\tilde\gamma_2(\ell)|\geq |\gamma_1(0)\gamma_1(\ell)|$. This follows immediately from the classical Schur theorem \cite[Thm.\ 5.1]{sullivan2008} if the majorization does not increase curvature, i.e.,
$
\tau(\tilde\gamma(I))\leq \tau(\gamma(I))$ for all intervals $I\subset[0,\ell].
 $
 Thus it remains only to show:

 \begin{proposition}\label{prop:tau}
 For every $\C^1$ curve $\gamma\colon[0,\ell]\to M$   there exists a chord-convex $\C^1$ proper majorization $\tilde\gamma\colon[0,\ell]\to\R^2$ which does not increase curvature.
  \end{proposition}
  
 We will show below (Lemma \ref{prop:tildegamma2}) that when $\gamma$ is $\C^{1,1}$, then \emph{every} majorization $\tilde\gamma$ is curvature nonincreasing. This may also be true in the $\C^1$ case, which would generalize the above proposition, but we do not need that stronger result here.

\subsection{Curvature-chord estimates}\label{subsec:estimates}
To prove Proposition \ref{prop:tau} we need the following estimates for the chord-length  of curves in $M^n$.
Given a nonnegative quantity $\rho \to 0$, we use the standard notation $f=O(\rho^m)$ if $|f|\le C\rho^m$ for some constant $C>0$ and small $\rho$, and $f=o(\rho^m)$ if $f/\rho^m\to 0$. These estimates are understood componentwise for vector-valued quantities.  The following  result holds in all Riemannian manifolds.

\begin{lemma}\label{lem:2point}
Choose normal coordinates
centered at a point $o$ of $M^n$. Then for $a,b\in T_oM\simeq\R^n$
$$
\big|\exp_o(a)\exp_o(b)\big|^2
=
|ab|^2-\frac13\,R(a,b,b,a)+O\!\left((|a|+|b|)^5\right).
$$
\end{lemma}

\begin{proof}
Let $x\colon[0,1]\to\R^n$ be the coordinate representation of the geodesic segment $\exp_o(a)\exp_o(b)$ with
$
x(0)=a, x(1)=b.
$
Then
$$
\big|\exp_o(a)\exp_o(b)\big|^2=\int_0^1 g_{x(t)}(\dot x,\dot x)\,dt .
$$
In normal coordinates the metric
$
g_{ij}(x)=\delta_{ij}-\frac13 R_{ikjl}(o)x^k x^l+O(|x|^3),
$
which implies
$
g_x(u,u)=|u|^2-\frac13 R(x,u,u,x)+O(|x|^3|u|^2).
$
Set $\rho:=|a|+|b|$. If $|a|$ and $|b|$ are sufficiently small, the geodesic ball $B\subset M$ of radius $\rho$ centered at $o$ is convex. Therefore $\exp_o(a)\exp_o(b)\subset B$. It follows that $|x|=O(\rho)$ and $|\dot x|=O(\rho)$. Hence
$$
g_{x(t)}(\dot x,\dot x)
=
|\dot x|^2-\frac13 R(x,\dot x,\dot x,x)+O(\rho^5).
$$
The geodesic equation states that
$
\ddot x^k+\Upgamma^k_{ij}(x)\dot x^i\dot x^j=0,
$
where $\Upgamma^k_{ij}(x)=O(|x|)$ since we are in normal coordinates. Since $|x|=O(\rho)$ and $|\dot x|=O(\rho)$, it follows that
$
\ddot x=O(\rho^3).
$
Set $v:=b-a$, $\ell(t):=a+tv$, and $\eta(t):=x(t)-\ell(t)$. Then
$\ddot\eta=\ddot x=O(\rho^3)$. 
Since $\eta(0)=\eta(1)=0$, integration gives
$\dot\eta=O(\rho^3)$, and $\eta=O(\rho^3)$.
Since 
$
|\dot x|^2=|v+\dot\eta|^2
=
|v|^2+2\langle v,\dot\eta\rangle+|\dot\eta|^2,
$
$$
\int_0^1 |\dot x|^2\,dt
=
|v|^2+2\big\langle v, \eta(1)-\eta(0)\big\rangle+\int_0^1 |\dot\eta|^2 dt
=
|ab|^2+0+O(\rho^6).
$$
Since $x=\ell+\eta=\ell+O(\rho^3)$  and $\dot x=v+\dot\eta=v+O(\rho^3)$,  the multilinearity of $R$ gives
$$
\int_0^1 R(x,\dot x,\dot x,x)\,dt
=
\int_0^1 R(\ell,v,v,\ell)\,dt+O(\rho^6).
$$
Finally we compute that 
$$
R(\ell,v,v,\ell)=R(a+tv,v,v,a+tv)=R(a,v,v,a)=R(a,b-a,b-a,a)=R(a,b,b,a),
$$
which completes the proof.
\end{proof}

Using the last lemma we next obtain the key relation we need between geodesic curvature $\kappa$ and chord length of curves, which again holds in all Riemannian manifolds. 

\begin{lemma}\label{prop:riemann}
 Let $\gamma\colon[0,\ell]\to M^n$ be a $\C^1$ unit speed curve and suppose that $t\in(0,\ell)$ is a twice differentiable point of $\gamma$.
Then, for sufficiently small $h>0$
$$
|\gamma(t-h)\gamma(t+h)|
=
2h-\frac{\kappa(t)^2}{3}h^3+o(h^3).
$$
\end{lemma}

\begin{proof}
Choose normal coordinates centered at $o:=\gamma(t)$. Since $\gamma$ has unit speed, we
have $\langle \nabla_{\gamma'(t)}\gamma', \gamma'(t)\rangle = 0$. Thus we may assume that
$$
\gamma'(t)=e_1,\qquad \nabla_{\gamma'(t)}\gamma'=\kappa(t)e_2.
$$
Set $\ol\gamma(s):=\exp_o^{-1}(\gamma(t+s))$. Since the Christoffel symbols vanish at the origin in normal
coordinates,
$
\kappa_0:=|\ol\gamma''(0)| = |\nabla_{\gamma'(t)}\gamma'| = \kappa(t),
$
and we have
$$
\ol\gamma(h)-\ol\gamma(-h)
=
\Bigl(he_1+\frac{\kappa_0}{2}h^2e_2+o(h^2)\Bigr)
-
\Bigl(-he_1+\frac{\kappa_0}{2}h^2e_2+o(h^2)\Bigr)
=
2he_1+o(h^2).
$$
Recall that $g_x(u,u)=|u|^2-\frac13R(x,u,u,x)+O(|x|^3|u|^2)$. Since $\ol\gamma(s)=se_1+O(s^2)$, and
$\ol\gamma'(s)=e_1+O(s)$, we have $R(\ol\gamma(s), \ol\gamma'(s), \ol\gamma'(s), \ol\gamma(s))=O(s^4)$. Also
$|\ol\gamma(s)|^3|\ol\gamma'(s)|^2=O(s^3)$. Therefore
$$
1 = g_{\ol\gamma(s)}(\ol\gamma'(s), \ol\gamma'(s)) = |\ol\gamma'(s)|^2 + O(s^3).
$$
Now write $\ol\gamma'(s)=e_1+\kappa_0s\,e_2+r(s)$, where $r(s)=o(s)$. Then
$$
|\ol\gamma'(s)|^2 = 1 + 2\langle r(s), e_1\rangle + \kappa_0^2s^2 + o(s^2).
$$
Since $|\ol\gamma'(s)|^2 = 1+O(s^3)$, it follows that $\langle r(s), e_1\rangle = -\kappa_0^2s^2/2 + o(s^2)$. Hence
$$
A(h):=\langle \ol\gamma(h)-\ol\gamma(-h), e_1\rangle
=
\int_{-h}^h \langle \ol\gamma'(s), e_1\rangle\, ds
=
2h-\frac{\kappa_0^2}{3}h^3+o(h^3).
$$
Since $\ol\gamma(h)-\ol\gamma(-h)=2he_1+o(h^2)$, its component orthogonal to $e_1$ is $o(h^2)$. Therefore
$
|\ol\gamma(h)\ol\gamma(-h)|^2 = A(h)^2 + o(h^4),
$
and consequently
$$
|\ol\gamma(h)\ol\gamma(-h)|
=
2h-\frac{\kappa_0^2}{3}h^3+o(h^3).
$$
Now apply Lemma \ref{lem:2point} to $a=\ol\gamma(-h)$ and $b=\ol\gamma(h)$. Since $\ol\gamma(h)=he_1+O(h^2)$, we have
$R(\ol\gamma(-h), \ol\gamma(h), \ol\gamma(h), \ol\gamma(-h))=O(h^6)$, since the $h^4$ and $h^5$ terms vanish by the skew-symmetry of $R$. Thus
$$
|\gamma(t-h)\gamma(t+h)|^2
=
|\ol\gamma(-h)\ol\gamma(h)|^2 + O(h^5)
=
4h^2-\frac{4\kappa_0^2}{3}h^4+o(h^4).
$$
Taking square roots completes the proof.
\end{proof}

In addition to the pointwise estimate obtained above we also need the following weaker but uniform  chord-length estimate. We say that a curve is $\C^{1,1}$ if it is $\C^{1,1}$ in some (and therefore every) local coordinate chart.
 
\begin{lemma}\label{lem:gamma2}
Let $\gamma\colon[0,\ell]\to M^n$ be a  $\C^{1,1}$ unit speed curve. Then
there exist constants $C>0$ and $h_0>0$ such that for every $t\in (0,\ell)$ and every $h\in(0,h_0)$ with $t\pm h\in[0,\ell]$,
$$
|\gamma(t-h)\gamma(t+h)| \ge 2h - C h^3 .
$$
\end{lemma}

\begin{proof}
Since
$\gamma([0,\ell])$ is compact and $M$ is smooth,  in normal
coordinates centered at $o:=\gamma(t)$ we have the uniform estimates
$$
g_{ij}(x)=\delta_{ij}+O(|x|^2),\qquad \Upgamma_{ij}^k(x)=O(|x|),
$$
for all points $x$ in a ball $B_{h_0}(0)\subset T_oM$ independent of $t$.
Let
$
\ol \gamma(s):=\exp_o^{-1}(\gamma(t+s)).
$
Since $\gamma$ has unit speed, we have
$
|\overline{\gamma}(s)| = |\gamma(t)\gamma(t +s)| \le |s|.
$
Furthermore, 
we may choose $h_0$ so small that  for all $|s| < h_0$, we have the uniform bound
$
|\overline{\gamma}'(s)|^2 \le 2.
$
Since $\ol\gamma$ is $\C^{1,1}$,  $\nabla_{\overline{\gamma}'}\overline{\gamma}'$ is essentially bounded. In normal coordinates centered at $o$,
$$
(\nabla_{\overline{\gamma}'(s)}\overline{\gamma}')_k
=
\overline{\gamma}''_k(s)+\Upgamma_{ij}^k(\overline{\gamma}(s))\overline{\gamma}_i'(s)\overline{\gamma}_j'(s).
$$
Since the Christoffel symbols are uniformly bounded on $B_{h_0}(0)$, as is $|\ol\gamma'|$, it
follows that $\overline{\gamma}''$ is essentially bounded. Consequently, if  $e_1 := \overline{\gamma}'(0)$, we have
$$
\overline{\gamma}'(s)=\overline{\gamma}'(0)+\int_0^s\ol\gamma''(u)\,du=e_1+O(s).
$$

Next, since $|\overline{\gamma}(s)| \le |s|$, we have
$g_{\overline{\gamma}(s)}(\cdot,\cdot)=|\cdot|^2+O(|\overline{\gamma}(s)|^2)|\cdot|^2=|\cdot|^2+O(s^2)|\cdot|^2$.
It follows that $1=|\overline{\gamma}'(s)|^2 +O(s^2)|\overline{\gamma}'(s)|^2$. Since $|\overline{\gamma}'(s)|^2 \le 2$, we conclude that
$$
|\overline{\gamma}'(s)|^2 = 1 + O(s^2).
$$

 Now writing $\overline{\gamma}'(s)=e_1+\eta(s)$ with $|\eta(s)|=O(s)$, we have
$|\overline{\gamma}'(s)|^2=1+2\langle \eta(s),e_1\rangle+|\eta(s)|^2$.
It follows that $\langle \eta(s),e_1\rangle=O(s^2)$. So
$\overline{\gamma}_1'(s):=\langle \overline{\gamma}'(s),e_1\rangle=1+O(s^2)$.
Integration yields
$
\overline{\gamma}_1(h)-\overline{\gamma}_1(-h)=2h+O(h^3).
$
So
$$
|\overline{\gamma}(h)\overline{\gamma}(-h)| \ge |\overline{\gamma}_1(h)-\overline{\gamma}_1(-h)|\ge 2h-Ch^3,
$$
where $C$ is independent of $t$ since the preceding $O(\cdot)$-estimates  are independent of $t$, due to uniform bounds on $g_{ij}$, $\Upgamma_{ij}^k$, $|\ol \gamma'|$, and $|\nabla_{\ol \gamma'}\ol\gamma'|$.
Since $K\leq 0$ in $M$, the exponential map is noncontracting. So
$
|\gamma(t-h)\gamma(t+h)|
\ge |\ol \gamma(h)\ol \gamma(-h)|\ge 2h-Ch^3,
$
as desired. 
\end{proof}

\subsection{The $\C^{1,1}$ case}\label{subsec:C11}
Using the preceding estimates, we now establish the following result which yields Proposition \ref{prop:tau} in the $\C^{1,1}$ case:
 
\begin{lemma}\label{prop:tildegamma2}
Let $\gamma\colon[0,\ell]\to M^n$ be a  $\C^{1,1}$ curve, and $\tilde\gamma\colon[0,\ell]\to\R^2$ be any chord-convex curve which majorizes $\gamma$. Then $\tilde\gamma$ is also $\C^{1,1}$, and the geodesic curvatures $\tilde\kappa\leq \kappa$ almost everywhere.
\end{lemma}
\begin{proof}
By Lemma \ref{lem:gamma2} and the majorization property
$
|\tilde\gamma(t-h)\tilde\gamma(t+h)|\ge 2h-Ch^3,
$
for every $t\in (0,\ell)$ and small $h$.
 Set
$
u:=\tilde\gamma(t+h)-\tilde\gamma(t),
$
$
v:=\tilde\gamma(t)-\tilde\gamma(t-h).
$
Since $\tilde\gamma$ has unit speed, 
$|u|, |v|\le h$. Thus
$$
|u-v|^2=2|u|^2+2|v|^2-|u+v|^2
\le 4h^2-(2h-Ch^3)^2
=4Ch^4-C^2h^6
\le C h^4,
$$
where we relabel the constants $C$.
Letting $D_h^2 \phi(t):=\phi(t+h)-2\phi(t)+\phi(t-h)$ be the second symmetric difference, we obtain
$
|D_h^2\tilde\gamma(t)|
=
|u-v|
\le 
C h^2 .
$
Then for any coordinate function $x$ of $\tilde\gamma$
$$
|D_h^2x(t)|\le C h^2 .
$$
Define $\alpha(t):=x(t)+Ct^2/2$ and $\beta(t):=x(t)-Ct^2/2$. Then
$D_h^2\alpha(t)\ge0$ and $D^2_h\beta(t)\le0$. Thus  $\alpha$
is convex and $\beta$ is concave.
So $\alpha'$ is nondecreasing and
$\beta'$ is nonincreasing. It follows that
$-C(t-s)\le x'(t)-x'(s)\le C(t-s)$. Hence $x'$ and consequently $\tilde\gamma'$ is Lipschitz, as desired.

Next, to obtain the inequality between geodesic curvatures, fix $t\in(0,\ell)$ such that both $\tilde\gamma$ and $\gamma$ are twice differentiable at
$t$. By Lemma \ref{prop:riemann},
$$
2h-\frac{\tilde\kappa(t)^2}{3}h^3+o(h^3)
=|\tilde\gamma(t-h)\tilde\gamma(t+h)|
\geq
|\gamma(t-h)\gamma(t+h)|
=2h-\frac{\kappa(t)^2}{3}h^3+o(h^3)
.
$$
Thus
$
\tilde\kappa(t)\le \kappa(t),
$
since the geodesic curvatures are nonnegative. 
\end{proof}

\subsection{The general case}\label{subsec:proof}
Finally, we extend the $\C^{1,1}$ case to the $\C^1$ case by an approximation to finish the proof of Proposition \ref{prop:tau} and thereby obtain Theorem \ref{thm:schur}. For every $\C^1$ curve $\gamma\colon[0,\ell]\to\R^2$, with nonvanishing speed, there exists a continuous function $\theta_\gamma\colon[0,\ell]\to\R$, defined up to $2k\pi$, $k\in\mathbf{Z}$, such that 
$$
\gamma'(t)=|\gamma'(t)|\big(\cos(\theta_\gamma(t)),\sin(\theta_\gamma(t))\big).
$$
 We call $\theta_\gamma$ the \emph{turning angle}  of $\gamma$. Note that $\theta_\gamma$ is monotone  when $\gamma$ is convex.

\begin{proof}[Proof of Proposition \ref{prop:tau}]
Note that $\tau(\gamma)$ is the limit of the sum of the exterior angles  of  polygonal geodesic
approximations of $\gamma$. Rounding the corners of these polygonal curves, we obtain a family of
$\C^{1,1}$ curves $\gamma_i\colon[0,\ell]\to M$ with constant speed $v_i\to 1$ and the same endpoints
as $\gamma$, such that $\length(\gamma_i)\to\length(\gamma)$ and
$
\tau(\gamma_i(I))\to \tau(\gamma(I))
$
for every interval $I\subset[0,\ell]$.
It follows from Lemma \ref{lem:reshetnyak} that there are chord-convex curves $\widetilde\gamma_i\colon[0,\ell]\to\mathbf{R}^2$ with constant speed $v_i$ which properly majorize $\gamma_i$. By Lemma \ref{prop:tildegamma2},  $\widetilde\gamma_i$
are $\C^{1,1}$. We may assume that
$\widetilde\gamma_i(0)=(0,0)$ and $\widetilde\gamma_i(\ell)=(-d,0)$, for $d\geq 0$, and
$\tilde\gamma_i$ lies above the $x$-axis. 
Let $\theta_i:=\theta_{\tilde\gamma_i}$ be the turning angle of
$\widetilde\gamma_i$. So
$$
\widetilde\gamma_i(t)=\int_0^t v_i\big(\cos\theta_i(s),\sin\theta_i(s)\big)ds,
$$
 where $v_i$ is the speed of $\widetilde\gamma_i$. Since $\widetilde\gamma_i$ is chord-convex, $\theta_i$ is
nondecreasing and we may assume that $0\le \theta_i\le 2\pi$. Thus, by
Helly's selection theorem \cite{rudin1976}, after passing to a subsequence  $\theta_i(t)\to\theta(t)$
for every $t\in[0,\ell]$, where $\theta\colon[0,\ell]\to\mathbf{R}$ is a nondecreasing (and possibly discontinuous) function.
Now define
$$
\widetilde\gamma(t):=\int_0^t \big(\cos\theta(s),\sin\theta(s)\big)\,ds.
$$
Since  $\theta$ is nondecreasing, $\widetilde\gamma$ is convex.
Furthermore, since  $v_i\to 1$, the
dominated convergence theorem yields that $\widetilde\gamma_i\to\widetilde\gamma$ uniformly on
$[0,\ell]$. In particular,
$
\widetilde\gamma(0)=(0,0)
$
and
$
\widetilde\gamma(\ell)=(-d,0).
$
Further, for every $a,b\in[0,\ell]$,
$$
|\widetilde\gamma(a)-\widetilde\gamma(b)|
=
\lim_{i\to\infty}
|\widetilde\gamma_i(a)-\widetilde\gamma_i(b)|
\ge
\lim_{i\to\infty}
|\gamma_i(a)\gamma_i(b)|
=
|\gamma(a)\gamma(b)|.
$$
Thus $\widetilde\gamma$ majorizes $\gamma$. It follows that $\widetilde\gamma$ is differentiable, which in turn yields that it is $\C^1$ by convexity, as we showed in the proof of Lemma \ref{lem:reshetnyak}. Finally we check that $\theta$ is the turning angle of $\tilde\gamma$. Indeed since $\widetilde\gamma$ is $\C^1$, it has a well-defined turning angle $\vartheta\colon[0,\ell]\to[0,2\pi]$. Further,
$
\widetilde\gamma'(t)=(\cos\theta(t),\sin\theta(t))
$
for almost every $t$. On the other hand,
$
\widetilde\gamma'(t)=(\cos\vartheta(t),\sin\vartheta(t))
$
for every $t$. Since $\theta,\vartheta\in[0,2\pi]$, it follows that $\theta=\vartheta$
a.e. Since $\theta$ is nondecreasing and $\vartheta$ is continuous, we conclude that
$\theta=\vartheta$ everywhere.

Now let $I=[a,b]\subset[0,\ell]$. Since $\theta_i$ and $\theta$ are the turning angles of $\widetilde\gamma_i$ and $\widetilde\gamma$, 
$$
\tau(\widetilde\gamma(I))
=
\theta(b)-\theta(a)
=
\lim_{i\to\infty}\big(\theta_i(b)-\theta_i(a)\big)
=
\lim_{i\to\infty}\tau(\widetilde\gamma_i(I)).
$$
Furthermore, by Lemma \ref{prop:tildegamma2},
$
\tau(\widetilde\gamma_i(I))\le \tau(\gamma_i(I)).
$
Hence
$$
\tau(\widetilde\gamma(I))
=
\lim_{i\to\infty}\tau(\widetilde\gamma_i(I))
\le
\lim_{i\to\infty}\tau(\gamma_i(I))
=
\tau(\gamma(I)),
$$
which completes the proof.
\end{proof}
\section{Rigidity of Convex Bodies}
Here we establish a  rigidity result for the boundaries of convex bodies in Cartan-Hadamard manifolds. This result was established in \cite[Prop.\ 4.4]{ghomi2025-convexity} for $\C^2$ boundaries. Using Theorem \ref{thm:schur}, we generalize that result to the $\C^{1}$ case. 

\begin{proposition}\label{prop:nested}
Let  $C\subset  M^n$ and $C'\subset \R^n$  be convex bodies with $\C^{1}$ boundaries $\Gamma$ and $\Gamma'$ respectively. Suppose that there exists a $\C^1$ isometry $f\colon\Gamma\to\Gamma'$ which preserves the total curvature of $\C^1$ curves. Then $f$ extends to an isometry $C\to C'$. 
\end{proposition}
\begin{proof}
For any $x\in\Gamma$ set $x':=f(x)$.  
By the proof of \cite[Lem.\ 4.2]{ghomi2025-convexity}, it suffices to show that $f$ preserves extrinsic distances or chord lengths, i.e., for every pair of points $x$, $y\in\Gamma$, 
$
|xy|= |x'y'|
$ 
(this follows from the generalized Kirszbraun extension theorem due to Lang-Schroeder \cite{lang-schroeder1997,akp2011}).

Let $\Pi\subset\R^n$ be a plane containing $x'y'$ which is transverse to $\Gamma'$. Then $\gamma':=\Pi\cap\Gamma'$ is a closed $\C^1$ convex curve in $\Pi$.  Let $\arc{x'y'}$ be one of the arcs connecting $x'$, $y'$ in $\gamma'$, and $\arc{xy}:=f^{-1}(\arc{x'y'})$ be the corresponding arc in $\gamma:=f^{-1}(\gamma')$. 
By assumption, $f\colon \arc{xy}\to\arc{x'y'}$ preserves the total curvature of all subsegments, as well as the arclength.
The curvature  of $\arc{x'y'}$ in $\R^n$ is the same as its curvature  in $\Pi$, which we may identify with $\R^2$. Thus, by Theorem \ref{thm:schur}, 
$$
|xy|\geq |x'y'|,
$$
for all pairs of points $x$, $y\in\Gamma$.
On the other hand, by Reshetnyak's theorem, there exists a convex curve $\gamma''\subset \Pi$ which majorizes $\gamma$. But, by the above inequality, all chords of $\gamma$ are at least as long as the corresponding chords of $\gamma'$.
Thus $\gamma''$ majorizes $\gamma'$.  Since $\gamma''$ and $\gamma'$ are both convex planar curves, it follows that they are congruent \cite{brooks-strantzen1992}. So $\gamma'$ majorizes $\gamma$, which yields
$$
|xy|\leq |x'y'|
$$ 
and completes the proof.
\end{proof}

Combining Propositions \ref{prop:nested} and \ref{prop:C11} immediately yields the following result which had been obtained earlier  in the $\C^3$ case \cite{ghomi2025-convexity}.

\begin{theorem}\label{thm:extension}
Let $\Gamma\subset M^3$ be a $\C^1$ convex surface. Suppose that $K$ vanishes on tangent planes of $\Gamma$. Then $K$ vanishes on the convex body bounded by $\Gamma$.
\end{theorem}

This result  implies that if the curvature of a Cartan-Hadamard manifold $M^3$ vanishes outside a compact set $X$, then it vanishes everywhere, for we may let $\Gamma$ be a convex hypersurface enclosing $X$. Results of this type, which are sometimes called ``gap theorems'' \cite{seshadri2009}, were first obtained by Greene-Wu \cite{greene-wu1982} and Gromov \cite[Sec.\ 3]{ballmann-gromov-schroeder}. See \cite{ghomi-spruck2023a,schroeder-strake1989a} for more results in this genre, and \cite{ghomi-stavroulakis2026b} for a recent variation.

\section{Proof of Theorem \ref{thm:main}}
We are now ready to establish the main result of this work. First we need one more definition. For any $\C^{1,1}$ surface $\Gamma\subset M^3$, the Gauss-Kronecker curvature is defined almost everywhere by Rademacher's theorem and thus the \emph{total (signed) curvature} is given by
$$
\G(\Gamma):=\int_{\Gamma}GK.
$$
If $\Gamma$ is convex, let $d\colon M\to\R$ be the distance function from the convex body bounded by $\Gamma$. Then the \emph{outer parallel surface} $\Gamma^t$ of $\Gamma$ at distance $t> 0$ is defined as $d^{-1}(t)$. Note that $\Gamma^t$ is $\C^{1,1}$ \cite[Prop.\ 2.7]{ghomi-spruck2022} and thus $\G(\Gamma^t)$ is well-defined. We set
$$
\G(\Gamma):=\lim_{\;\;t\to 0^+}\G(\Gamma^t).
$$
This limit exists since $t\mapsto \G(\Gamma^t)$ is nondecreasing \cite[Cor.\ 5.3]{ghomi-spruck2022} and $\G(\Gamma^t)\geq 0$ by convexity of the distance function \cite{bridson-haefliger1999}.

\subsection{The inequality}
To establish \eqref{eq:main} let $\Gamma_0:=\partial\conv(\Gamma)$ be the boundary of the convex hull of $\Gamma$, and $\Gamma_0^t$ denote the outer parallel surfaces of $\Gamma_0$. 
By Gauss' equation 
$$
GK_{\Gamma_0^t}(p)=K_{\Gamma_0^t}(p)-K(T_p\Gamma_0^t),
$$ 
for almost all $p\in \Gamma_0^t$, where $K_{\Gamma_0^t}$ is the intrinsic curvature of $\Gamma_0^t$. Integrating this equation over $\Gamma_0^t$, and using the Gauss-Bonnet theorem, we obtain
$$
\G(\Gamma_0^t)= 4\pi-\int_{p\in\Gamma_0^t} K(T_p\Gamma_0^t)\geq 4\pi.
$$
Thus 
$
\G(\Gamma_0)\geq4\pi.
$ 
Furthermore, we have $\G(\Gamma\cap\Gamma_0)=\G(\Gamma_0)$, as was observed by Kleiner \cite[p.\ 42--43]{kleiner1992}, see \cite[Prop.\ 6.6]{ghomi-spruck2022}.
Let
$\Gamma_+\subset\Gamma$ be the region where $GK_\Gamma\geq 0$, and define $\G_+(\Gamma):=\int_{\Gamma_+} GK_\Gamma$ to be the \emph{total positive curvature} of $\Gamma$. Then
$$
\tilde\G(\Gamma)\geq \G_+(\Gamma)\geq \G(\Gamma\cap\Gamma_0)=\G(\Gamma_0)\geq 4\pi,
$$
as desired.

\subsection{The case of equality}
To characterize the equality case in \eqref{eq:main}, suppose that $\tilde{\G}(\Gamma)= 4\pi$. Then equalities hold in the last displayed expression. In particular,
$$
4\pi
=
\G(\Gamma_0)
=
\lim_{t\to 0} \G(\Gamma_0^t)
=
4\pi- \lim_{t\to 0} \int_{p\in \Gamma_0^t} K(T_p\Gamma_0^t).
$$
So $\int_{p\in \Gamma_0^t} K(T_p\Gamma_0^t)\to 0$. This forces $K$ to vanish on tangent planes of $\Gamma_0$, which are well-defined by Proposition \ref{prop:convexhull}. Indeed, suppose towards a contradiction that $K(T_{p_0}\Gamma_0)<0$ for some $p_0\in\Gamma_0$. Then there exists a neighborhood $U\subset \Gamma_0$ of $p_0$ and a constant $\delta>0$ such that
$
K(T_p\Gamma_0)\le -2\delta
$
for all $p\in U$.
 Let $U_t\subset\Gamma_0^t$ be the neighborhoods obtained by the outward normal flow of $U$, and note that 
 $$
 \operatorname{area}(U_t)\geq \operatorname{area}(U)
 $$ 
 since the nearest point projection $\Gamma_0^t\to\Gamma_0$ is nonexpansive \cite[Cor.\ 2.5]{bridson-haefliger1999}. Then for all sufficiently small $t$, we have $
K(T_p\Gamma_0^t)\le -\delta
$
for all $p\in U_t$, since $\Gamma_0^t\to\Gamma_0$ in $\C^1$, $K$ is continuous, and $\Gamma_0$ is compact.
Hence
$$
\int_{p\in\Gamma_0^t} K(T_p\Gamma_0^t)
\leq
\int_{p\in U_t} K(T_p\Gamma_0^t)
\leq
-\delta\,\operatorname{area}(U_t)\leq-\delta\,\operatorname{area}(U),
$$
which is impossible, since  $\int_{p\in \Gamma_0^t} K(T_p\Gamma_0^t)\to 0$. So $K=0$ on tangent planes of $\Gamma_0$, as claimed.
Now by Theorem \ref{thm:extension}, $\conv(\Gamma)$ is flat and therefore is isometric to a convex body in $\R^3$. Under this isometry, $\Gamma$ corresponds to a closed surface $\Gamma'\subset\R^3$ with $\tilde\G(\Gamma')=\tilde\G(\Gamma)=4\pi$.
Hence, by Chern-Lashof's theorem, $\Gamma'$ is convex \cite[Thm.\ 3]{chern-lashof1957}. Therefore $\Gamma$ is convex. It follows that $\Gamma=\Gamma_0$, which completes the proof.

\subsection{Generalizations}
Here we improve the regularity requirement in Theorem \ref{thm:main}, and  discuss other refinements which extend the notion of tightness in Euclidean space.

\subsubsection{$\C^{1,1}$ surfaces}
In the argument above, we used smoothness of $\Gamma$ only to invoke Chern-Lashof's theorem. Since that result also holds for $\C^{1,1}$ surfaces,  so does Theorem \ref{thm:main}.
Indeed, let $f\colon \Gamma\to \R^3$ be a $\C^{1,1}$ immersion, for any closed surface $\Gamma$. Then the corresponding unoriented Gauss map
$
\ol\nu\colon \Gamma\to \RP^2
$
 is Lipschitz. Note also that $\ol\nu=\pi\circ \nu$ for any local Gauss map $\nu$, where $\pi\colon \S^2\to \RP^2$ is the covering map. So at every twice differentiable point of $f$,  the (unsigned) Jacobian  $\text{Jac}(\ol\nu)=\text{Jac}(\nu)=|GK|$ since $\pi$ is a local isometry. Hence,
by the area formula \cite[Thm.\ 3.2.3]{federer1969},
$$
\tilde \G(\Gamma)=\int_\Gamma \text{Jac}(\ol\nu)
=\int_{\RP^2} \#\big(\ol\nu^{\,-1}\big)\geq 4\pi,
$$
since the area of $\RP^2$ is $2\pi$, and $\#(\ol\nu^{\,-1})\geq2$. Indeed, if for $u\in\S^2$ we let $[u]:=\{u,-u\}\in\RP^2$, then
$\ol\nu^{\,-1}([u])$ is  the set of critical
points of the height function 
$
h_u(p):=\langle f(p),u\rangle,
$
which includes a minimum and a maximum by compactness of $\Gamma$.
Now if $\tilde \G(\Gamma)=4\pi$, then $\#(\ol\nu^{\,-1})=2$ almost everywhere, which yields that $h_u$ has exactly two critical points for almost every $u\in\S^2$. Then $\Gamma$ is homeomorphic to $\S^2$, by the generalized Reeb theorem \cite[Thm.\ 1']{milnor1964} \cite{mcauley1972}, and $f$ is a convex embedding by a result of Kuiper \cite[Thm.\ 4]{kuiper1970}.

\subsubsection{Total positive curvature and tightness}
  The proof of Theorem \ref{thm:main} actually shows that the total positive curvature $\G_+(\Gamma)\geq 4\pi$ with equality only if $K$ vanishes on the convex hull of $\Gamma$.
 Since  $\G_+(\Gamma)\leq\tilde\G(\Gamma)$, this is an improvement of the theorem.  Closed surfaces $\Gamma\subset\R^3$ with $\G_+(\Gamma)= 4\pi$ are called \emph{tight} \cite{cecil-chern1985}, since they minimize $\tilde\G$ in their topological class. Adopting the same terminology for Cartan-Hadamard manifolds, we may say that
\emph{all tight surfaces in $M^3$ lie in flat convex bodies}. 

We should also mention that for a closed surface $\Gamma$ of topological genus $g$  in $\R^3$, $\tilde{\G}(\Gamma)\geq 2\pi(2+2g)$ with equality precisely when $\Gamma$ is tight; however, Solanes \cite{solanes2007} constructed examples  in hyperbolic space $\mathbf{H}^3$ with $\tilde{\G}(\Gamma)< 2\pi(2+2g)$ for every genus $g\geq 1$. Thus  in Cartan-Hadamard manifolds minimizers of $\G_+$ and $\tilde \G$ no longer coincide, and $\G_+$ is indeed the right quantity for extending the notion of tightness.

\section*{Acknowledgements}
We thank Igor Belegradek, Fran\c{c}ois Fillastre, Misha Gromov, Ivan Izmestiev,  Reza Pakzad, Anton Petrunin, and  Andrzej \'{S}wi\k{e}ch for useful communications.

\bibliography{references}

\end{document}